\documentclass[a4paper,12pt]{amsart}

\usepackage{euscript,eufrak,verbatim}
\usepackage{graphicx}
\usepackage[usenames]{color}
\usepackage[colorlinks,linkcolor=red,anchorcolor=blue,citecolor=blue]{hyperref}
\usepackage[psamsfonts]{amssymb}

\newtheorem{thm}{Theorem}[section]
\newtheorem{prop}[thm]{Proposition}
\newtheorem{lem}[thm]{Lemma}

\def\Zp{\mathbb{Z}_{p}}
\def\Qp{\mathbb{Q}_{p}}
\def\Zm{\mathcal{Z}_{\widehat{\mu_\Lambda}}}
\def\Q{\mathbb{Q}}
\def\Z{\mathbb{Z}}
\def\ZE{\mathcal{Z}_{\widehat{\mu_E}}}
\def\m{\mathfrak{m}}
\def\Card{{\rm Card}}

\begin{document}
	
\title[Fuglede's conjecture holds in $\Qp$]{Fuglede's conjecture holds in $\Qp$ }

\author
{Aihua Fan}
\address
{
 LAMFA, UMR 7352 CNRS, Universit\'e de Picardie,
33 rue Saint Leu, 80039 Amiens, France}
\email{ai-hua.fan@u-picardie.fr}

\author
{Shilei Fan}
\address
{School of Mathematics and Statistics \& Hubei Key Laboratory of Mathematical Sciences, Central  China Normal University, Wuhan, 430079,   China }
\email{slfan@mail.ccnu.edu.cn}

\author{Lingmin Liao}
\address{LAMA, UMR 8050, CNRS,
Universit\'e Paris-Est Cr\'eteil, 61 Avenue du
G\'en\'eral de Gaulle, 94010 Cr\'eteil Cedex, France}
\email{lingmin.liao@u-pec.fr}

\author
{Ruxi Shi}
\address
{LAMFA, UMR 7352 CNRS, Universit\'e de Picardie,
33 rue Saint Leu, 80039 Amiens, France}
\email{ruxi.shi@u-picardie.fr}

\thanks{A. H. FAN was supported by NSF of China (Grant No. 11471132) and self-determined research funds of CCNU (Grant No. CCNU14Z01002); S. L. FAN was supported by NSF of China (Grant Nos. 11401236 and 11231009)}

\begin{abstract}
Fuglede's conjecture in  $\Q_p$ is proved. That is to say, a Borel set of positive and finite Haar measure in $\Qp$ is a spectral set if and only if it tiles $\Q_p$ by translation. 
\end{abstract}

\maketitle

\section{Introduction}
Let $G$ be a locally compact abelian group and $\widehat{G}$ be its dual group. Denote by $\mathfrak{m}$ (sometimes by $dx$) the Haar measure on $G$.  Consider a Borel measurable subset $\Omega$ in $G$ with $0<\mathfrak{m}(\Omega)<\infty$.
We say that $\Omega$ is a {\em spectral set} if there exists a set $\Lambda \subset \widehat{G}$ which is
an orthonormal basis of the Hilbert space $L^2(\Omega)$. Such a set $\Lambda$ is called a {\em spectrum} of $\Omega$ and $(\Omega,\Lambda)$ is called a {\em spectral pair}.
A set $\Omega$ is said to be a
{\em  tile}  of $G$ by translation if there exists a  set $T \subset G$ of translates
such that $\sum_{t\in T} 1_\Omega(x-t) =1$ for almost all $x\in G$, where $1_A$ denotes the indicator function of a set $A$. The set $T$ is  then called a
{\em tiling complement} of $\Omega$ and $(\Omega,T)$ is called a {\em tiling pair}. 

When $G=\mathbb{R}^n$, Fuglede \cite{F}  formulated the following conjecture: 
{\em 	A Borel set $\Omega\subset \mathbb{R}^n$ of positive and finite Lebesgue measure is a spectral set if and only if it is a tile.}
We could formulate a general Fuglede's conjecture for a locally compact abelian group $G$, simply called Fuglede's conjecture in $G$: {\em 
	A Borel set $\Omega\subset G$ of positive and finite Haar measure is a spectral set if and only if it is a tile.
} In its generality, this generalized conjecture is far from
proved. We could rather ask for which groups it holds.  The question arises even for finite groups.  
\medskip

Let $p\ge 2$ be prime and $\mathbb{Q}_p$ be the field of $p$-adic numbers.   
In the present paper, we will prove the following theorem
that  Fuglede's conjecture in $\mathbb{Q}_p$ holds.

\begin{thm}\label{mainthm}
	Assume that  $\Omega\subset \Qp$ is a  Borel set of positive and finite Haar measure. Then $\Omega$ is a spectral set if and only if it is a tile of $\Qp$. In this case,  $\Omega$ must be an almost compact open set.
\end{thm}

By {\em  almost compact open set} we mean a Borel set $\Omega\subset \Q_p$ such that there exists a compact open set $\Omega^{\prime}$ satisfying
$$\m(\Omega\setminus \Omega^{\prime})=\m(\Omega^{\prime}\setminus \Omega)=0.$$

In \cite{FFS}, the first assertion of Theorem \ref{mainthm} was proved under the additional assumption that $\Omega$ is a compact open set in $\mathbb{Q}_p$, and furthermore, the compact open spectral sets were characterized by their  $p$-homogeneity (see the definition of $p$-homogeneity in Section 1 of \cite{FFS}). 

\begin{thm}[\cite{FFS}, Theorem 1.1]\label{compactopen}
	Let $\Omega $ be a compact open set in $\mathbb{Q}_p$. The following
	statements  are equivalent: \\
	\indent {\rm (1)} \ $\Omega$ is a spectral set; \\
	\indent {\rm (2)} \ $\Omega$ tiles $\mathbb{Q}_p$;\\
	\indent {\rm (3)} \ $\Omega$ is $p$-homogenous.
\end{thm}

Fuglede \cite{F} proved the conjecture in $\mathbb{R}^d$ under the extra assumption that the spectrum or the tiling complement is a lattice of $\mathbb{R}^d$. There are many positive results under different extra assumptions before the work \cite{TT} where  Tao gave a counterexample: there exists a spectral subset of $\mathbb{R}^d$ with $d\ge 5$ which is not a tile. After that, Matolcsi \cite{M}, Matolcsi and Kolountzakis \cite{KM,KM2}, Farkas and Gy \cite{FG}, Farkas, Matolcsi and   M{\'o}ra \cite{FMM06} gave a series of counterexamples  which show that
both directions of Fuglede's conjecture fail in $\mathbb{R}^d( d\geq 3)$.  
However,  the conjecture  is still open in low dimensions $d=1,2$.


As pointed out in Proposition 3.1 of \cite{Fan},  which follows \cite{JP1998}, that  $(\Omega,\Lambda)$ is a spectral pair in $\Q_p$   is equivalent to
\begin{equation}\label{SpectralSet}
  \forall  \xi\in \widehat{\Q}_p,   \qquad \sum_{\lambda \in \Lambda } |\widehat{1_{\Omega}}|^2(\xi -\lambda) = \m(\Omega)^2,
\end{equation}
where $\widehat{1_{\Omega}}$ is the Fourier transform of $1_{\Omega}$. By definition, $(\Omega,T)$ is a tiling pair of $\Q_p$ means that
\begin{equation}\label{Tile}
\sum_{t \in T } 1_{\Omega}(x -t) = 1, \quad \m\text{-a.e. } x\in \Qp
\end{equation}

A subset $E$ of $\Q_p$ is said to be {\em uniformly discrete} if $E$ is countable and $\inf_{x,y \in E} |x-y|_p>0$, where $|\cdot|_p$ denotes the $p$-adic absolute value on $\Q_p$. Remark that if $E$ is uniformly discrete, then  ${\rm Card}(E\cap K)<\infty$  for any compact subset $K$ of $\Q_p$ so that  
 \begin{equation}\label{mu_E}
 	\mu_E = \sum_{\lambda\in E} \delta_\lambda
 \end{equation}
 defines a  discrete
 Radon measure.  Observe that both (\ref{SpectralSet}) and (\ref{Tile})
are of the form 
\begin{equation}\label{EQ1}
\mu_E * f =w, 
\end{equation}where $f$ is a non-negative integrable function and $w>0$ is a positive number. Actually,
both spectrum $\Lambda$ and  tiling complement $T$ in $\Q_p$ are  uniformly discrete (see Proposition \ref{uniformdiscrete}).
The above convolution equation (\ref{EQ1})
will be our main concern. 
\medskip

Our proof of Theorem \ref{mainthm} will be based on   the analysis of the set of zeros  of the Fourier transforms $\widehat{\mu_{\Lambda}}$ and  $\widehat{\mu_{T}}$, where both $\mu_\Lambda$ and $\mu_T$ are considered as Bruhat-Schwartz distributions and their Fourier transforms
$\widehat{\mu_{\Lambda}}$ and  $\widehat{\mu_{T}}$ are also Bruhat-Schwartz distributions.

Actually, we will first  prove that  tiles of  $\Q_p$ are almost compact open sets.
Then the direction  ``tiling $\Rightarrow$ spectral" of Theorem \ref{mainthm} follows from Theorem \ref{compactopen}. Next we will  directly prove ``spectral $\Rightarrow$ tiling".  Consequently,  spectral sets and tiles are all almost compact open. 

In \cite{FF}, it is proved that any {\em bounded} tile of $\Q_p$ is almost  compact open. The proof there is different from the one in the present paper.  Without using the theory
of distribution, that proof is more direct and easily understandable. But the boundedness is assumed as an extra condition.  

The article is organized  as follows. In Section 2, we present preliminaries on the field $\Q_p$ of $p$-adic numbers and on the $\mathbb{Z}$-module generated by the $p^n$-th roots of unity.  Some useful facts from  the theory of Bruhat-Schwartz  distributions are also presented. It is proved that a particular structure is shared by the set of zeros of $\widehat{\mu_E}$ for all uniform discrete sets $E$.    In Section 3, we study  the functional equation $\mu_E * f = 1$ under the assumption that $E$ is uniformly discrete and $f$ is a non null and non-negative
integrable function. Such a relation has strong constraints on $f$ and on $E$. For example, we prove that $\widehat{f}$ has compact support. Section 4 is devoted to the proof that tiles in $\Q_p$ are all almost compact open and spectral sets in $\Q_p$ are all almost bounded. It is then proved  that tiles in $\Q_p$ are spectral sets. The proof that spectral sets in $\Q_p$ are tiles is given in Section 5.   In the last section, we show the dual  Fuglede's conjecture  in $\Q_p$ and  give some remarks on the case of higher dimensions.

\setcounter{equation}{0}

\section{Preliminaries}

Our study on the Fuglede's conjecture on $\Q_p$ is strongly related to the following functional equation
\begin{equation}\label{F-equation}
\mu_E * f = 1
\end{equation}
where $E\subset \Q_p$ is a uniformly discrete set and $f \in L^1(\Q_p)$ is a non-negative integrable function such that $\int_{\Q_p}f d\m>0$. 

The convolution in (\ref{F-equation})  is understood as a convolution of Bruhat-Schwartz distributions and even as a convolution in the Colombeau algebra of generalized functions. One of reasons is that the Fourier transform of the infinite Radon measure $\mu_E$ is not defined for the measure $\mu_E$ but for the distribution $\mu_E$. Both spectral set and tile are characterized by special cases of the  equation
 (\ref{F-equation}).
In fact, (\ref{SpectralSet})   means that a spectral pair $(\Omega, \Lambda)$ is characterized by (\ref{F-equation})
with $f=\m(\Omega)^{-2} |\widehat{1_{\Omega}}|^2$ and $E=\Lambda$; 
(\ref{Tile})  means that a tiling pair 
$(\Omega, T)$ is characterized by (\ref{F-equation})
with $f=1_{\Omega}$ and $E=T$.

In this section, after having presented some basic facts like the Lebesgue density theorem, the uniform discreteness of spectrum and of tiling complement etc,  we will recall some result from \cite{s} on  the $\mathbb{Z}$-module generated by $p^n$-th roots of unity,  which is a key for our study.
 Then we will present some useful facts from  the theory of Bruhat-Schwartz
distributions and from the theory of the Colombeau algebra of generalized functions.   At the end, we will investigate the set of zeros
of the Fourier transform $\widehat{\mu_E}$.

\subsection{The field of $p$-adic number $\Qp$}\label{p-adicfield}

We start with a quick recall of $p$-adic numbers. 
Consider the field $\mathbb{Q}$ of rational numbers and a prime $p\ge 2$.
Any nonzero number $r\in \mathbb{Q}$ can be written as
$r =p^v \frac{a}{b}$ where $v, a, b\in \mathbb{Z}$ and $(p, a)=1$ and $(p, b)=1$
(here $(x, y)$ denotes the greatest common divisor of the two integers $x$ and $y$). 
 We define 
$|r|_p = p^{-v_p(r)}$ for $r\not=0$ and $|0|_p=0$.
Then $|\cdot|_p$ is a non-Archimedean absolute value. That means\\
\indent (i)  \ \ $|r|_p\ge 0$ with equality only when $r=0$; \\
\indent (ii) \ $|r s|_p=|r|_p |s|_p$;\\
\indent (iii) $|r+s|_p\le \max\{ |r|_p, |s|_p\}$.\\
The field $\mathbb{Q}_p$ of $p$-adic numbers is the completion of $\mathbb{Q}$ under
$|\cdot|_p$. The ring $\mathbb{Z}_p$ of $p$-adic integers is the set of $p$-adic numbers with absolute value $\leq 1$. A typical element $x$ of $\mathbb{Q}_p$ is of the form
\begin{equation}\label{HenselExp}
x= \sum_{n= v}^\infty a_n p^{n} \qquad (v\in \mathbb{Z}, a_n \in \{0,1,\cdots, p-1\} \text{ and } a_v\neq 0). 
\end{equation}
Here, $v_p(x):=v$ is called the $p$-{\em valuation} of $x$. 

 A non-trivial additive character on $\mathbb{Q}_p$ is defined by
$$
\chi(x) = e^{2\pi i \{x\}}
$$
where $\{x\}= \sum_{n=v_p(x)}^{-1} a_n p^n$ is the fractional part of $x$ in (\ref{HenselExp}). From this character we can get all characters $\chi_\xi$ of $\mathbb{Q}_p$, by defining 
$\chi_\xi(x) =\chi(\xi x)$. We remark that 
\begin{align}\label{one-in-unit-ball}
\chi(x)=e^{2\pi i k/p^n}, \quad  \text{  if }  x \in \frac{k}{p^n}+\mathbb{Z}_p  \ \  (k, n \in \Z), 
\end{align}
and
\begin{align}\label{integral-chi}
\int_{p^{-n}\mathbb{Z}_p} \chi(x)dx=0 \ \text{ for all } n\geq 1.
\end{align}

The map $\xi \mapsto \chi_{\xi}$ from $\Q_p$ to $\widehat{\Q}_p$ is an isomorphism. We write $\widehat{\Q}_p\simeq \Qp$ and identify a point $\xi \in \Q_p$ with the point $\chi_\xi \in \widehat{\Q}_p$.  
For more information on $\mathbb{Q}_p$ and $\widehat{\Q}_p$, the reader is referred to the book \cite{Vvz}.

The following notation will be used in the whole paper.

\medskip
\noindent {\bf Notation}:
\begin{itemize}
\item $\mathbb{Z}_p^\times := \mathbb{Z}_p\setminus p\mathbb{Z}_p=\{x\in \mathbb{Q}_p: |x|_p=1\}$,
the group of units of $\mathbb{Z}_p$.


\item $B(0, p^{n}): = p^{-n} \mathbb{Z}_p$,  the (closed) ball centered at $0$ of radius $p^n$.

\item $B(x, p^{n}): = x + B(0, p^{n})$. 

\item $ S(x, p^{n}): = B(x, p^{n})\setminus B(x, p^{n-1})$,  a ``sphere".

\item $ \mathbb{L} : = \{\{x\}: x\in \Q_p\}$, a complete set of representatives of the cosets of the additive subgroup $\mathbb{Z}_p$.

\item $ \mathbb{L}_n : = p^{-n} \mathbb{L}$. 
\end{itemize}
\medskip

The Lebesgue density theorem   holds in $\Q_p$ for the Haar measure. 
 
\begin{prop}[\cite{PT}]\label{Lebesgue}
Let $\Omega\subset \Qp$ be a bounded Borel set such that  $\mathfrak{m}(\Omega)>0$. Then 
$$\lim_{n\to \infty}\frac{\mathfrak{m}(B(x,p^{-n})\cap \Omega) }{\mathfrak{m}(B(x, p^{-n}))}=1,  \quad  \m\hbox {-a.e. }  x \in \Omega.$$
\end{prop}








\subsection{Fourier Transformation}
The Fourier transformation of  $f\in L^1(\Q_p)$ is defined to be 
$$\widehat{f}(\xi)=\int_{\Qp}f(x)\overline{\chi_\xi(x)} dx   \quad (\forall \xi\in \widehat{\Q}_p\simeq \Qp).$$

A complex function $f$ defined on $\Q_p$ is called \textit{uniformly locally constant} if  there exists $n\in \mathbb{Z}$ such that 
\[f(x+u)=f(x) \quad \forall x\in \Q_p, \forall u \in B(0, p^n).\] 
The following proposition shows that for  an integrable function $f$, having compact support and being uniformly locally constant are dual properties for $f$ and its Fourier transform.
 
\begin{prop}\label{2case}
	Let $f\in L^1(\Q_p)$ be  a complex-value integrable function. \\
	\indent {\rm (1)} If $f$ has compact support, then $\widehat{f}$ is uniformly locally constant.\\
   \indent {\rm (2)} If $f$ is uniformly locally constant, then $\widehat{f}$ has compact support.
\end{prop}
\begin{proof}
\indent {\rm (1)}	Suppose that $f$ is supported by $B(0,p^n)$. For any $x\in \Q_p$ and any $u \in B(0, p^{-n})$, we have
	\begin{align*}
		\widehat{f}(\xi+u)-\widehat{f}(\xi)
		&=\int_{B(0,p^n)}f(y)\overline{\chi(\xi y)}(\overline{\chi(uy)-1})dy.
	\end{align*}
	Notice that if $y\in B(0,p^n)$, we have $|uy|_p\le 1$. So by (\ref{one-in-unit-ball}), $$\chi(uy)-1=0.$$
	Therefore, $\widehat{f}(\xi+u)-\widehat{f}(\xi)=0$  for
	all $u \in B(0, p^{-n})$. Thus
         $\widehat{f}$ is  uniformly  locally constant. 
	
   \indent {\rm (2)} Suppose  that $f(x+u)=f(x)$ for any $x\in \Q_p$ and any $u \in B(0, p^{n})$. Observing $$\Q_p= \mathbb{L}_{n}+B(0, p^{n})=\bigcup_{z\in \mathbb{L}_{n}} B(z, p^{n}),$$ we deduce
	$$
	\widehat{f}(\xi)=\sum_{z \in \mathbb{L}_{n}}\int_{B(z, p^{n})}f(y)\overline{\chi(\xi y)}dy
	=\sum_{z \in \mathbb{L}_{n}}f(z)\int_{B(z, p^{n})}\overline{\chi(\xi y)}dy.
	$$
By (\ref{integral-chi}), we have $\int_{B(z, p^{n})}\chi(\xi y)dy=0$ for $|\xi|_p>p^{-n} $.  Therefore, $\widehat{f}(\xi) =0$ for $|\xi|_p >p^{-n}$.
\end{proof}	
The above proposition is the key which will allow us to prove that a tile  $\Omega$ in  $\Q_p$ is an almost compact open set by showing that the support of $\widehat{1_{\Omega}}$ is compact (see Proposition \ref{CompactSupport}). 
Consequently, a tile is a spectral set by Theorem \ref{compactopen}.


Now we  prove that the spectra  and tiling complements are always uniformly discrete.

\begin{prop}\label{uniformdiscrete} Let $\Omega\subset \Q_p$ be a Borel set of positive and finite Haar measure.\\
\indent{\rm (1)} If  $(\Omega,\Lambda)$ is a spectral pair, then $\Lambda$ is uniformly discrete. \\
\indent{\rm (2)} If $(\Omega,T)$ is a tiling pair, then $T$ is uniformly discrete. 
\end{prop}
\begin{proof}{\rm (1)} By the fact  $\widehat{1_{\Omega}}(0)=\m(\Omega)>0$ and the continuity of  the function $\widehat{1_{\Omega}}$, there exists an integer $n_0$ such that 
$\widehat{1_{\Omega}}(x)\neq 0$ for all $x\in B(0, p^{n_0})$. 
This, together with  the orthogonality  
    $$\widehat{1_{\Omega}}(\lambda-\lambda^{\prime})=0\quad \forall \ \lambda, \lambda^{\prime} \in \Lambda \ \text{distinct }, $$ 
    implies that  $|\lambda-\lambda^{\prime}|_p>p^{n_0}$ for different $\lambda,\lambda^{\prime}\in \Lambda$.


\indent{\rm (2)}  Consider the continuous function on $\mathbb{Q}_p$ defined by
$$
f(x):= 1_\Omega *1_{-\Omega}(x)=
\int_{\mathbb{Q}_p} 1_\Omega(y)1_\Omega(y-x)dy
=\m(\Omega \cap (\Omega+x)).
$$
The fact $f(0)=\m(\Omega)>0$ and the continuity of $f$ imply  that  there exists an integer $n_0$ such that 
   \begin{equation}\label{OO}
   	\m(\Omega\cap(\Omega+t))>0, \quad \text {for }  t \in B(0,p^{n_0}).
   	\end{equation}
For different $t\in T$ and $t^{\prime}\in T$, the tiling property
implies $\m((\Omega+t)\cap(\Omega+t^{\prime}))=0$. So  
$$\m(\Omega\cap (\Omega+t-t^{\prime}))
= \m((\Omega+t)\cap(\Omega+t^{\prime}))=0$$
by the translation invariance of $\m$.   
Thus we must have $|t-t^{\prime}|_p >p^{n_0}$, by (\ref{OO}). 
\end{proof}

\subsection{$\mathbb{Z}$-module generated by $p^n$-th roots of unity} \label{Zmodule}

Let $m\ge 2$ be an integer and let  $\omega_m = e^{2\pi i/m}$, which is a primitive $m$-th root of unity. Denote  by $\mathcal{M}_m$  the set of integral points
$(a_0, a_1, \cdots, a_{m-1}) \in \mathbb{Z}^m$ such that
$$
\sum_{j=0} ^{m-1} a_j \omega_m^j =0.
$$
The set $\mathcal{M}_m$ is clearly a $\mathbb{Z}$-module. In the following  we assume that $m=p^n$ is a power of a prime number. 

\begin{lem} [\cite{s}, Theorem 1]\label{SchLemma}
If $(a_0,a_1,\cdots, a_{p^n-1})\in  \mathcal{M}_{p^n}$,
	then for any integer $0\le i\le p^{n-1}-1$ we have $a_i=a_{i+jp^{n-1}}$ for all $j=0,1,\dots, p-1$.
\end{lem}

Lemma \ref{SchLemma} has the following two special forms. The first one is an immediate consequence.
\begin{lem}
\label{permu}
	Let $(b_0,b_1,\cdots, b_{p-1})\in \mathbb{Z}^{p}$.
	If $\sum_{j=0}^{p-1} e^{2\pi i b_j/p^n}=0$, then subject to a permutation of $(b_0,\cdots, b_{p-1})$, there exist $0\leq r \leq p^{n-1}-1$ such  that
	$$b_j  \equiv r+ jp^{n-1} (\!\!\!\!\mod p^n) $$ for all  $ j =0,1, \cdots,p-1$.
\end{lem}
\begin{lem}[\cite{FFS}, Lemma 2.5]\label{C}   
	Let $C$ be a finite subset of $\mathbb{Z}$. If $\sum_{c\in C} e^{2\pi i c/p^n}=0$, then $p \mid {\rm Card} (C)$ and $C$ is decomposed into  ${\rm Card}( C) /p$ disjoint subsets
	$C_1,  C_2,  \cdots, $ $ C_{{\rm Card} (C)/p}$, such that each $C_j$ consists of $p$ points and
	$$\sum_{c\in C_j}e^{2\pi i c/p^n}=0.$$
\end{lem} 
Now applying Lemmas \ref{permu} and \ref{C}, we have the following lemma which will be useful in the paper.
\begin{lem}\label{zmod}
Let $C\subset\Q_p$ be a finite set.\\
\indent{\rm(1)} If $\sum_{c\in C}\chi(c)=0,$ then $p\mid {\rm Card}(C)$ and  $\sum_{c\in C}\chi(x  c)=0$
	for any $x \in \mathbb{Z}_p^\times$ (i.e. $|x|_p=1$).\\
\indent{\rm(2)} If there exists $\xi\in \Qp$ such that 
	$\sum_{c\in C}\chi(\xi c)=0,$ 
	then for any $c\in C$, there exists  $c^{\prime}\in C$ such that 	$|c-c^{\prime}|_p=p/|\xi|_p.$\\
\indent{\rm(3)}	
If there exists a finite set $\mathbb{I}\subset \Z$ 
such that 
	\begin{align}\label{keyeq}
	\sum_{c\in C}\chi(p^i c)=0 \quad  \text{for all } i\in \mathbb{I},
	\end{align}
then $p^{{\rm Card}(\mathbb{I})} \mid  \Card(C).$ 
\end{lem}

\begin{proof}{\rm(1)} It is a direct consequence of Lemmas \ref{permu} and \ref{C}. See \cite[Lemma 2.6]{FFS} for  details.\\
\indent{\rm(2)} Let $C=\{c_1,c_2,\cdots, c_m \}$. Recall that $\chi(\xi c)=e^{2\pi i\{\xi c\}}$.  There exist an integer $n$ and a subset $\{n_1, n_2 \cdots, n_m\}$ of $\Z$ such that  
$$\chi(\xi c_k)=e^{2\pi i n_k/p^n} , \quad  k=1,2,\cdots, m.$$ 
By Lemma \ref{C},  $p\mid \Card(C)$ and $C$  is decomposed  into $\Card(C)/p$ disjoint subsets $C_0, C_1,\cdots,C_{\Card(C)/p}$, such that each $C_j$ consists of $p$ points and
	$$\sum_{c\in C_j}\chi(\xi c)=0.$$ Without loss of generality, assume that $c\in C_0$. By  Lemma \ref{permu}, we have 
	$$|c-c^{\prime}|_p=\frac{p}{|\xi|_p}, \quad \text{ if  } c^{\prime} \in C_{0}\setminus\{c\}.$$\\
\indent {\rm(3)} Assume that  $\mathbb{I}=\{i_1,i_2,\cdots,i_n \}$ with $i_1<i_2<\cdots<i_n$.
 By Lemma \ref{C} and the equality (\ref{keyeq}) with $i=i_1$, we have $p\mid \Card(C)$ and  $C$ can be decomposed into  $ \Card(C)/p$ subsets $C_1, C_2, \cdots, C_{ \Card(C)/p}$ such that
	 each $C_j$ consists of $p$ points and $$\sum_{c \in C_j} \chi(p^{i_1} c) =0.$$
	By Lemma \ref{permu}, if  $c$ and  $ c^{\prime}$ lie in the same $C_j$, we have
	 \begin{align}\label{ceq1}
	 	|c-c^{\prime}|_p\leq  p^{1+i_1}.
	 \end{align}
	  

	 Now we consider the  equality (\ref{keyeq}) when $i=i_2$.  Since $i_{1}< i_{2} $,    (\ref{ceq1}) and  (\ref{one-in-unit-ball}) imply that the function  $$c\mapsto \chi(p^{i_2} c)$$ is constant on each $C_j$. From each $C_j$, take one element  $\widetilde{c_j}$. Let $\widetilde{C}$ be the set consisting  of these
	 $\widetilde{c_j}$. 
	  

	  Since each $C_j$ contains $p$ elements,  the equality (\ref{keyeq}) with $i=i_2$ is equivalent to 
	 $$
	 \sum_{\widetilde{c} \in \widetilde{C}} \chi(p^{i_2}c) =0.
	 $$
	 By Lemma \ref{C},  $p \mid \Card( \widetilde{C} )$, which implies $p^2\mid  \Card(C)$.
	  
	 
	  By induction,  we get $p^n\mid   \Card(C)$. \\
\end{proof}

\subsection{Bruhat-Schwartz distributions in $\Q_p$}\label{subsec2.4}
Here we give a brief description of the theory following 
\cite{Aks,t,Vvz}. 
  Let $\mathcal{E}$ denote the space of the uniformly locally constant functions.
The space $\mathcal{D}$ of {\em Bruhat-Schwartz test functions} is, by definition, constituted of uniformly locally constant functions
 of compact support. Such a test function $f\in \mathcal{D}$ is a finite linear combination of indicator functions of the form $1_{B(x,p^k)}(\cdot)$, where $k\in \mathbb{Z}$ and $x\in \Q_p$. The largest of such numbers $k$
 is  denoted by $\ell:= \ell(f)$ and is  called the {\em parameter of constancy} of $f$. Since $f\in \mathcal{D}$ has compact support, the minimal number $\ell':=\ell'(f)$ such that the support of $f$ is contained in $B(0, p^{\ell'})$ exists and will be called the {\em parameter of compactness} of $f$.

 Clearly, $\mathcal{D}\subset \mathcal{E}$. The space $\mathcal{D}$ is provided with a topology of topological vector space as follows: a sequence $\{\phi_n \}\subset \mathcal{D}$ is called a {\em null sequence} if there is a fixed pair of $l, l^{\prime}\in \mathbb{Z}$ such that each $\phi_n$ is constant on every ball of radius $p^l$ and is supported by the ball $B(0,p^{l^{\prime}})$ and the sequence $\phi_n$ tends uniformly to zero.
 
 A {\em Bruhat-Schwartz distribution} $f$ on $\Q_p$ is by definition a continuous linear functional on $\mathcal{D}$. The value of $f$ at $\phi \in \mathcal{D}$
 will  be denoted by $\langle f, \phi \rangle$.  
 Note that linear functionals on $\mathcal{D}$ are automatically continuous. This property allows us to easily construct distributions. Denote by
 $\mathcal{D}'$ the space of  Bruhat-Schwartz  distributions. The space $\mathcal{D}'$ is equipped with the weak topology induced by  $\mathcal{D}$.
 
  A locally integrable function $f$ is considered as a distribution:  for any $\phi \in \mathcal{D}$,
$$
\langle f,\phi\rangle=\int_{\Q_p} f\phi d\m.
$$
The discrete measure $\mu_E$ defined by (\ref{mu_E}) is also a distribution:
for any $\phi \in \mathcal{D}$,
$$
\langle \mu_E,\phi\rangle= \sum_{\lambda\in E} \phi(\lambda).
$$
Here for each $\phi$, the sum is finite  because $E$ is uniformly discrete and each ball contains at most a finite number of points in $E$.
Since the test functions in $\mathcal{D}$ are  uniformly locally constant  and have compact support,
the following proposition is a direct consequence of Proposition \ref{uniformdiscrete} or of the fact (see also \cite[Lemma 4]{Fan}) that
\begin{equation}\label{FB}
\widehat{1_{B(c, p^k)}}(\xi) = \chi(-c \xi) p^{k} 1_{B(0, p^{-k})}(\xi).
\end{equation}

\begin{prop}[\cite{t}, Chapter II 3]
 The Fourier transformation  $f \mapsto \widehat{f}$ is a  homeomorphism  from $\mathcal{D}$ onto $\mathcal{D}$.
  \end{prop}

 The {\em Fourier transform of a distribution} $ f\in \mathcal{D}'$ is a new distribution  $ \widehat{f}\in \mathcal{D}'$ defined by the duality
$$
\langle\widehat{f},\phi\rangle=\langle f,\widehat{\phi}\rangle, \quad \forall \phi \in \mathcal{D}.
$$
The Fourier transformation $f\mapsto \widehat{f}$ is a homeomorphism of $\mathcal{D}'$ onto $\mathcal{D}'$ under the weak topology \cite[ Chapter II 3]{t}.

\subsection{Zeros of the Fourier transform  of  a discrete  measure}

Let $f \in \mathcal{D}'$ be a distribution in $\Q_p$.  A  point $x\in \Q_p$ is  called a {\em  zero} of  $f$ if there exists an integer $n_0$
  such that  $$ \langle f, 1_{B(y,p^{n})}\rangle=0, \quad \text {for all }  y\in B(x,p^{n_0})  \text{ and all integers  }  n\leq  n_0 \text.$$
Denote by  $\mathcal{Z}_f$ the set of all zeros of $f$.
Remark that 
 $\mathcal{Z}_f$ is  the maximal open set $O$ on which $f$ vanishes, i.e.
 $\langle f, \phi\rangle=0$ for all $\phi \in \mathcal{D}$ such that the support of $\phi$ is contained in $O$.
 The {\em support}  of a distribution $f$ is defined as  the complementary set of  $\mathcal{Z}_f$ and is denoted by  $\operatorname{supp}(f)$.

Let $E$ be a  uniformly discrete set in $\Q_p$. The following proposition characterizes the structure of  $\ZE$, the set of zeros of the
Fourier transform of the discrete  measure $\mu_E$. It is bounded and is a union of spheres centered at $0$.
\begin{prop}\label{zeroofE} 
Let $E$ be a  uniformly discrete set in $\Q_p$.\\
\indent {\rm (1)} 
If $\xi\in \ZE$,  then $S(0,|\xi|_p)\subset \ZE$.\\
\indent {\rm (2)} The set $\ZE$ is bounded.
\end{prop}
\begin{proof}
 First remark that by using (\ref{FB}) we get  
	\begin{equation}\label{FofMu}
		\langle \widehat{\mu_E}, 1_{B(\xi,p^{-n})} \rangle
		=\langle \mu_E, \widehat{1_{B(\xi,p^{-n})}} \rangle
		=p^{-n}\sum_{\lambda\in E \cap B(0, p^{n})} \overline{\chi(\xi \lambda)}.
	\end{equation}  
	This expression will be used several times.
	
	\indent {\rm (1)} 
By definition,  $\xi\in \ZE$
implies that
there exists an integer $n_0$ such that 
 $$\langle \widehat{\mu_E}, 1_{B(\xi,p^{-n})} \rangle=0 , \quad \forall~ n\geq n_0. $$
By (\ref{FofMu}), this is equivalent to 
 \begin{equation}\label{zeross}
 \sum_{\lambda\in E \cap B(0, p^{n})}\chi(\xi\lambda)=0 ,\quad \forall~ n\geq n_0. 
 \end{equation}	
For any  $\xi^{\prime}\in S(0,|\xi|_p)$,  we have $\xi' = u \xi$ for some $u\in \mathbb{Z}_p^\times$.  By Lemma \ref{zmod} (1) and  the equality (\ref{zeross}), we obtain
	$$
	\sum_{\lambda\in E \cap B(0, p^{n})}\chi(\xi^{\prime}\lambda)= \sum_{\lambda\in E \cap B(0, p^{n})}\chi(\xi\lambda)=0,\quad \forall~ n\geq n_0. 
	$$
	 Thus, again by (\ref{FofMu}), $\langle \widehat{\mu_E}, 1_{B(\xi',p^{-n})} \rangle=0 $ for $ n\geq n_0$. 
	 We have thus proved $S(0,|\xi|_p)\subset\ZE$.
%

 \indent {\rm (2)} Fix  $\lambda_0 \in E$. By the  discreteness of $E$,  there exists an integer $n_0$ such that 
$$\forall \lambda \in E\setminus \{\lambda_0\}, \quad |\lambda-\lambda_0|_p\geq p^{-n_0}.$$
We are going to show that  $\ZE\subset B(0,p^{n_0+1})$, which is equivalent to that $\xi \not \in \ZE$ when $|\xi|_p\geq p^{n_0+2}$. To this end, we will prove that for all integer $n$ large enough such that $\lambda_0 \in B(0, p^n)$, we have
$$\langle \widehat{\mu_E}, 1_{B(\xi,p^{-n})} \rangle \neq 0.$$
In fact, if this is not the case, then by (\ref{FofMu}), 
\[ \sum_{\lambda\in E \cap B(0, p^{n})}\chi(\xi\lambda)=0. \]
Thus, by Lemma \ref{zmod} (2), for the given $\lambda_0$, we can find $\lambda \in E \cap B(0, p^{n})$, such that $|\lambda-\lambda_0|_p=p/|\xi|_p <p^{-n_0}$,  a contradiction.

\end{proof}
Remark that $n_0$  in the above proof depends only on the structure of $E$.  Define 
\begin{align}\label{def-n_E}
n_E:=\max_{\substack {\lambda, \lambda^{\prime}\in E \\  \lambda\neq \lambda^{\prime}} } v_p(\lambda- \lambda^{\prime}).
\end{align}
According to the above proof of the second assertion of Proposition  \ref{zeroofE}, we immediately get 
\begin{align}\label{nE}
\ZE\subset   B(0,p^{n_E+1}).
\end{align}

\subsection{Convolution and  multiplication of   distributions}
Denote $$
\Delta_k:=1_{B(0,p^k)}, \quad
\theta_k:=\widehat{\Delta}_k=p^k \cdot 1_{B(0,p^{-k})}.
$$
Let $f,g\in \mathcal{D}'$
 be two distributions. We define the {\em convolution} of $f$ and $g$ by
$$
\langle f*g,\phi \rangle =\lim\limits_{k\to \infty}	 \langle f(x), \langle g(\cdot),\Delta_k(x)\phi(x+\cdot) \rangle \rangle,
$$
if the limit exists for all $\phi \in \mathcal{D}$. 

{
\begin{prop} [\cite{Aks}, Proposition  4.7.3  ] If $f\in \mathcal{D}^{\prime}$, then  $f*\theta_k\in \mathcal{E}$
with  the
parameter of constancy at least  $-k$.
\end{prop}
}

We define the
{\em  multiplication} of $f$ and $g$ by
$$
\langle f\cdot g,\phi \rangle =\lim\limits_{k\to \infty}	 \langle g, ( f*\theta_k)\phi  \rangle,
$$
if the limit exists for all $\phi \in \mathcal{D}$.        The above definition of convolution is
compatible with the usual convolution of two integrable functions and the definition of multiplication is compatible with  the usual  multiplication of two locally integrable functions.

The following proposition shows that both the convolution and the multiplication are commutative when they are well defined and the convolution of two distributions is well defined if and only if the multiplication 
of their Fourier transforms is well defined. 

\begin{prop} [\cite{Vvz}, Sections 7.1 and 7.5] \label{Conv-Mul} Let $f, g\in \mathcal{D}'$ be two distributions. Then\\
\indent {\rm (1)} \  If $f*g$ is well defined, so is $g*f$ and  $f*g=g*f$.\\
\indent {\rm (2)} \ If $ f\cdot g$ is well defined, so is  $g\cdot f$
 and  $ f\cdot g= g\cdot f$.\\
 \indent {\rm (3)} \ $f*g$ is well defined  if and only   $\widehat{f}\cdot \widehat{g}$ is well defined. In this case, we have
$
\widehat{f*g}=\widehat{f}\cdot\widehat{g}$ and
$ \widehat{f\cdot g}=\widehat{f}*\widehat{g}.
$
\end{prop}

The following proposition justifies an intuition. 

 \begin{prop} \label{zeroproduct}
Let  $f, g\in \mathcal{D}'$ be two distributions. If $\operatorname{supp}(f) \cap \operatorname{supp}(g) =\emptyset $, then $f\cdot g$ is well defined and 
$f\cdot g =0$.
\end{prop}
\begin{proof}
Let $\phi\in \mathcal{D}$ with parameter of  constancy   $\ell$ and parameter of compactness $\ell^{\prime}$. 
By the assumption $\operatorname{supp}(f) \cap \operatorname{supp}(g) =\emptyset $, the distance between
the two compact sets $\operatorname{supp}(f) \cap B(0,p^{\ell^{\prime}})$ and $\operatorname{supp}(g) \cap B(0,p^{\ell^{\prime}})$ is strictly positive.
So there exists  an integer $n\leq \ell'$  such that  for any
$x \in \operatorname{supp}(f) \cap B(0,p^{\ell^{\prime}})$, the ball $B(x,p^{n})$ is contained in $\mathcal{Z}_g\cap B(0,p^{\ell^{\prime}})$.
Therefore, the  compact open  set   $$X:=\bigcup_{x\in\operatorname{supp}(f) \cap B(0,p^{\ell^{\prime}})}B(x,p^{-n})$$ 
satisfies $\operatorname{supp}(f)\cap B(0,p^{\ell^\prime}) \subset X \subset   \mathcal{Z}_g\cap B(0,p^{\ell^{\prime}})$.
 
 Define $\phi_1(x) :=\phi(x) \cdot 1_X(x)$. 
It follows that  $$\operatorname{supp}(\phi-\phi_1)\ \subset \ B(0,p^{\ell^{\prime}}) \setminus X \ \subset \  \mathcal{Z}_{f}\cap B(0,p^{\ell^{\prime}}).$$
 Thus, we have
 \begin{align*}
  \langle f\cdot g,\phi \rangle = & \langle  f\cdot g,( \phi-\phi_1+\phi _1)\rangle =\langle  f\cdot g,\phi-\phi_1\rangle +\langle f\cdot g,  \phi _1\rangle \\
    = &\lim\limits_{k\to \infty}	 \langle f, ( g*\theta_k)(\phi-\phi_1)  \rangle +\lim\limits_{k\to \infty}	 \langle g, ( f*\theta_k)\phi_1  \rangle \\
    =&  0,
    \end{align*}
    where the existences of the  last two limits   are due to  
    $$\operatorname{supp} ((g*\theta_k)(\phi-\phi_1)) \subset \operatorname{supp}(\phi-\phi_1) \subset \mathcal{Z}_{f},$$
and $$\operatorname{supp} ((f*\theta_k)\phi_1) \subset \operatorname{supp}\phi_1 \subset \mathcal{Z}_{g}.$$
\end{proof}


The multiplication of some special distributions   has a simple form. That is the case for  the multiplication of a uniformly locally constant function and a distribution.

\begin{prop}[\cite{Vvz} Section 7.5, Example 2]\label{prod}
	Let $f\in \mathcal{E}$ and let $G\in \mathcal{D}'$. Then for any $\phi\in \mathcal{D}$, we have
	$
	\langle f\cdot G, \phi \rangle = \langle G, f\phi \rangle.
	$
\end{prop}

 For a distribution $f\in \mathcal{D}'$, we define its
 {\em regularization} by the sequence of test functions  (\cite[ Proposition  4.7.4]{Aks})
$$
\Delta_k \cdot(f*\theta_k) \in \mathcal{D}.
$$

The regularization of a distribution  converges  to the distribution  
with respect to the weak topology.

\begin{prop}[\cite{Aks} Lemma 14.3.1]\label{limit}
	Let $f$ be a distribution in $\mathcal{D}'$.
	Then $\Delta_k \cdot (f*\theta_k)\to f$ in $\mathcal{D}'$ as $k\to \infty$. Moreover, for any test function $\phi\in \mathcal{D}$ we have
	$$
	\langle \Delta_k \cdot(f*\theta_k), \phi \rangle=\langle f, \phi \rangle, \quad \forall k\ge \max\{-\ell,\ell^{\prime}\},
	$$
	where $\ell$ and $\ell^{\prime}$ are the parameter of constancy and the  parameter of compactness of the function $\phi$ defined in Subsection \ref{subsec2.4}. 
\end{prop}

This approximation of distribution by test functions allows us to construct a   space which is bigger than the space of distributions. This larger space is the Colombeau algebra, which will be presented below. Recall that in the space of Bruhat-Schwartz distributions, the convolution and the multiplication are not well defined for all couples of distributions.  But in the Colombeau algebra,  the convolution and the multiplication are well defined  and these two operations are associative. 

\subsection{Colombeau algebra of generalized functions}
Consider the set $\mathcal{P}: = \mathcal{D}^\mathbb{N}$ of all sequences $\{f_k\}_{k\in \mathbb{N}}$ of test functions.  We introduce an algebra
structure on $\mathcal{P}$, defining the operations component-wisely
$$
\{f_{k}\}+\{ g_{k}\}=\{f_{k}+g_{k}\},
$$
$$
\{f_{k}\}\cdot\{g_{k}\}=\{f_{k}\cdot g_{k}\},
$$
where $\{f_{k}\}, \{g_{k}\}\in \mathcal{P}$. 

 Let $\mathcal{N}$ be the sub-algebra of elements $\{f_k\}_{k\in \mathbb{N}}\in \mathcal{P}$  such that for any compact set $K\subset\Q_p$ there exists $N\in \mathbb{N}$ such that   $f_{k}(x)=0$ for all $k\geq N, x\in K$. Clearly, $\mathcal{N}$ is an ideal in the algebra $\mathcal{P}$.   
 
  Then we introduce the  {\em  Colombeau-type algebra}
$$
\mathcal{G}=\mathcal{P}/\mathcal{N}.
$$
 The equivalence class of sequences which defines an element in  $ \mathcal{G}$ will be denoted by $\mathbf{f}=[f_k]$, called a generalized function.
 
 For any $\mathbf{f}=[f_k], \mathbf{g}=[g_k]\in \mathcal{G}$, the addition   and multiplication  are defined as  
 $$
 \quad \mathbf{f}+\mathbf{g}=[f_k+g_k],  \quad \mathbf{f}\cdot\mathbf{g}=[f_k\cdot g_k].
$$

Obviously,  $(\mathcal{G}, +, \cdot)$     is an associative and commutative algebra.

\begin{thm}[\cite{Aks} Theorem 14.3.3]
	The map $f\mapsto \mathbf{f}=[\Delta_k (f*\theta_k)]$ from  $\mathcal{D}'$  to $\mathcal{G}$ is  a linear embedding.
\end{thm}

Each distribution $f\in \mathcal{D}'$ is embedded into 
$\mathcal{G}$ by the mapping which associates $f$ to the generalized function determined by the of $f$. 
Thus we obtain that the multiplication  defined on the $\mathcal{D}'$ is associative in the following sense.
\begin{prop}\label{associative}
 Let $f,g,h\in \mathcal{D}'$. If $(f\cdot g)\cdot h$ and $f\cdot (g\cdot h)$ are well defined as multiplications of distributions, we have  $$(f\cdot g)\cdot h =f\cdot (g\cdot h).$$

\end{prop}
\setcounter{equation}{0}
\section{Study on $\mu_E * f = 1$}
In this section, 
we will   study the following equation 
\begin{equation}\label{EQ}
\mu_E * f = 1
\end{equation}
where  $0\le f\in L^1(\Q_p)$ with  $\int_{\Q_p}fd\m>0$, and
$E$ is a uniformly discrete subset of $\Q_p$. 
Suppose that $(E, f)$ is a solution of (\ref{EQ}). We will investigate
the density of $E$ and even   the distribution of $E$,  and the supports of the Fourier transforms 
$\widehat{\mu_E}$ and $\widehat{f}$.


We say that a uniformly discrete set  $E$ has a {\em bounded density} if the following limit exists for some $x_0\in \Q_p$	
$$
D(E):=\lim\limits_{k\to \infty}\frac{{\rm Card}(B(x_0,p^k)\cap E)}{\m (B(x_0,p^k))},
$$
which is called {\em the density} of $E$. Actually, if the limit exists for some $x_0 \in \Q_p$, then
it exists for all $x\in \Q_p$ and the limit is independent of $x$. In fact, for any $x_0, x_1\in \Qp$, when $k$ is  large enough such that  $|x_0-x_1|_p<p^k$, we have $B(x_0, p^k)=B(x_1, p^k)$.

For a function $g: \Q_p \rightarrow \mathbb{R}$, denote
$$
   \mathcal{N}_g :=\{x \in \Q_p: g(x)=0\}.
$$
If $g\in C(\Q_p)$ is a continuous function, then $\mathcal{N}_g $ is a closed set and $ \mathcal{Z}_g$ is the set of interior points of $\mathcal{N}_g $.
  But, the support of $g$ as a continuous function is equal to the support of $g$ as  a distribution.
  
  The following theorem gets together some properties of the solution $(E, f)$ of the equation (\ref{EQ}), which will be proved in this section. 
 

\begin{thm} \label{M1} Let $0\le f\in L^1(\mathbb{Q}_p)$ with  $\int_{\mathbb{Q}_p}fd\m>0$,
	and $E$ be a uniformly discrete subset of $\mathbb{Q}_p$. 
	Suppose that the equation {\rm (\ref{EQ})} is satisfied by $f$ and $E$. Then\\
	       		\indent {\rm (1)}\
	       		The support of $\widehat{f}$ is compact.\\
	   \indent {\rm (2)} The set  $\mathcal{Z}_{\widehat{\mu_E}}$  is bounded  and it is the union of  the  punctured ball $B(0, p^{-n_f})\setminus \{0\}$ and some spheres. \\
	\indent {\rm (3)}\ The density $D(E)$ exists and equals to 
	$1/\int_{\Q_p} f d\m$. Furthermore, 
	there exists an integer $n_f\in \mathbb{Z}$
	such that  for all integers $n \geq n_f$ we hav
	       $$
	       \forall \xi \in \Q_p, \quad \mbox{\rm Card}(E\cap B(\xi, p^{n})) = p^{n} D(E).
	        $$ 
	        	
	     
\end{thm}

Theorem \ref{M1} (1) and (2) will be proved in $\S 3.1$, the distribution of $E$ will be discussed in $\S 3.2$ and the equality
$D(E)=1/ \int_{\Q_p} f d \m$ will be proved in $\S 3.3$.



\subsection{Compactness of $\mbox{\rm supp} (\widehat{f})$ and structure of  $\mathcal{Z}_{\widehat{\mu_E}}$}

Our discussion is based on $\widehat{f}\cdot\widehat{\mu_E}=\delta_0$, which is implied by
 $f*\mu_E=1$ (see Proposition \ref{Conv-Mul}).


\begin{prop}\label{mainlem}
	Let $g\in C(\Q_p)$ be a continuous function and let $G\in \mathcal{D}'$ be a distribution. Suppose that the product $H=g\cdot G$ is well defined. 	Then
		$$\mathcal{Z}_{H} \subset \mathcal {N}_{g} \cup \mathcal{Z}_{G}.$$
		Consequently, if $f*\mu_E=1$ with $f\in L^1(\mathbb{Q}_p)$, then
		$\Q_p \setminus \{0\} \subset \mathcal {N}_{\widehat{f}} \cup \mathcal{Z}_{\widehat{\mu_E}},$
		which is equivalent to
	\begin{align}\label{mainprop}
	\{\xi\in\Q_p: \widehat{f}(\xi) \neq 0  \} \setminus \{0\} \subset \mathcal{Z}_{\widehat{\mu_E}}.
	\end{align}

\end{prop}

\begin{proof}
The second assertion follows directly from the first assertion
because $\widehat{f}\cdot\widehat{\mu_E}=\delta_0$ and $\mathcal{Z}_{\delta_0} =\Q_p\setminus \{0\}$.  

We now prove the first assertion.
It suffices to prove $\mathcal{Z}_{H} \setminus \mathcal {N}_{g} \subset \mathcal{Z}_{G}.$
	Take an arbitrary test function $\phi\in \mathcal{D}$ such that $$\operatorname{supp}(\phi) \subset \mathcal{Z}_{H} \setminus \mathcal {N}_{g} .$$ Since $\operatorname{supp}(\phi) \subset \{x:g(x)\neq 0 \}$, we can define the function
\begin{equation*}
	h(x)=
\begin{cases}
	\frac{1}{g(x)}, & \mbox{\rm for}~ x \in \operatorname{supp}\phi,\\
	0, & \text{elsewhere}.
\end{cases}
\end{equation*}
Since $g$ is continuous, it is bounded away from $0$ on $\operatorname{supp}\phi$. So, the function $h$  is  bounded and compactly supported. Hence it belongs to $L^1(\Q_p)$.
On the other hand,  we have $h(x)g(x)=1_{\operatorname{supp}(\phi)}(x)\in \mathcal{D}$. Thus
\begin{align*}
\langle G, \phi \rangle	
&=\langle G, 1_{\operatorname{supp}(\phi)} \cdot \phi  \rangle
=\langle 1_{\operatorname{supp}(\phi)}\cdot G,\phi \rangle
=\langle (h\cdot g) \cdot G, \phi \rangle
\end{align*}
where we have used Proposition \ref{prod} for the second equality. 
 Notice that $\operatorname{supp}(h)\subset \mathcal{Z}_{H}$. By Proposition 
\ref{zeroproduct}, $h\cdot H$ is well defined and $h\cdot H=0$.
By the associativity of the multiplication (see Proposition \ref{associative}), we get 
\begin{align*}
	\langle G, \phi \rangle	
&=\langle h\cdot H,\phi  \rangle	
=0.
\end{align*}
 Thus we have proved $\mathcal{Z}_{H} \setminus \mathcal {N}_{g} \subset \mathcal{Z}_{G}.$
\end{proof}


Notice that $\widehat{f}(0)=\int_{\Q_p} f d \m>0$ and  $\widehat{f}$ is a continuous function. 
It follows that  there exists a small ball  where $\widehat{f}$ is non-vanish. 
Let
\begin{align} \label{nf} 
	n_{f}:=\min \{n\in\mathbb{Z}: \widehat{f}(x)\neq 0,\text{ if }  x\in B(0,p^{-n}) \}.
\end{align}

\begin{prop}\label{CompactSupport}
The Fourier transform $\widehat{f}$   has compact support.
The set $\mathcal{Z}_{\widehat{\mu_E}}$ is bounded and
\begin{equation} \label{3.1}
 B(0,p^{-n_f})\setminus \{0\}\subset \ZE.
\end{equation}
\end{prop}

\begin{proof} 	The boundedness of $\mathcal{Z}_{\widehat{\mu_E}}$ is already proved. See Proposition  \ref{zeroofE} (2). This together with (\ref{mainprop}) implies 
	the compactness of $\mbox{\rm supp} (\widehat{f})$. We get (\ref{3.1})
	immediately from (\ref{mainprop}).
\end{proof}


\subsection{Distribution of $E$}

 The uniformly discrete set $E$ involved in the equation (\ref{EQ}) shares the following uniform distribution property. 
 
\begin{prop}\label{Structure}
	The cardinality  of  $E\cap B(\xi,p^{n_f})$  is independent  of $\xi\in \mathbb{Q}_p$.
	Consequently,  the   set  $E$ admits  a bounded density $D(E)$.  Moreover,  for all integers $n\geq n_f$, we have
	\begin{equation}  \label{numberE}
 \forall \xi \in \Q_p, \quad	{\rm Card}(E\cap B(\xi,p^n))=  p^{n} D(E).
	\end{equation}
\end{prop}
\begin{proof}
	For simplicity, we denote
	 $$E_n^\xi:=E\cap B(\xi,p^n),$$ 
	and write $E_n:=E\cap B(0,p^n)$  when $\xi=0$. It suffices to prove 
	 $$\forall \xi \in \Q_p,  \quad\Card(E_{n_f})= \Card(E_{n_f}^{\xi}).$$
	 
For any given $\xi \in \Q_p$, let $k=-v_p(\xi)$. If $k\leq n_f$, then $E_{n_f}=E_{n_f}^{\xi}$. So obviously $\Card(E_{n_f})= \Card(E_{n_f}^{\xi})$.

Now we suppose $k> n_f$. Then for any $\eta \in B(0,p^{k}) \setminus B(0,p^{n_f} )$, $$B(\eta,p^{-k})\subset B(0,p^{-n_f})\setminus \{0\}.$$
By (\ref{3.1}) in Proposition \ref{CompactSupport}, we have   $\langle \widehat{\mu_E}, 1_{B(\eta,p^{-k})} \rangle =0$, which
by (\ref{FofMu}), is equivalent to 
%
\begin{equation}\label{eq1}
\sum_{\lambda\in E_k}{\chi(\eta \lambda)}=0.
\end{equation}
Taking  $\eta= p^{k-1}$ in (\ref{eq1}), we have 
$
\sum_{\lambda\in E_k}{\chi(p^{k-1}\lambda)}=0.
$
Observe that  $$B(0,p^k)= \bigsqcup_{i=0}^{p-1} B(ip^{-k},p^{k-1})$$ and that the function  $\chi(p^{k-1}\cdot)$ is locally constant  on each ball of radius $p^{k-1}$. 
So we have 
$$
0=\sum_{\lambda\in E_k}{\chi(p^{k-1}\lambda)}=\sum_{i=0}^{p-1} {\chi\Big(\frac{i}{p}\Big)}{\rm Card}\Big(E_{k-1}^{\frac{i}{p^k}}\Big).
$$
Applying 
Lemma \ref{SchLemma}, we obtain
\begin{equation}
\label{ud1}
\mbox{\rm Card}\Big(E_{k-1}^{\frac{i}{p^k}}\Big)={\rm Card}\Big(E_{k-1}^{\frac{j}{p^k}}\Big), 
\quad \forall \ 0\le i,  j \le p-1.
\end{equation}

%


 Similarly,   taking $\eta=p^{k-2}$ in  (\ref{eq1}), we have 
%
$$
0=\sum_{0\le i, j \le p-1}{\chi\Big(\frac{i}{p^2}+\frac{j}{p}\Big)} {\rm Card}\Big(E_{k-2}^{\frac{i}{p^{k}}+\frac{j}{p^{k-1}}}\Big).
$$
Again, Lemma \ref{SchLemma} implies 
$${\rm Card}\Big(E_{k-2}^{\frac{i}{p^{k}}+\frac{j}{p^{k-1}}}\Big)={\rm Card}\Big(E_{k-2}^{\frac{i}{p^{k}}+\frac{m}{p^{k-1}}}\Big)\quad \forall \ 0\le i,  j,m \le p-1.$$ 

Since $$\sum_{j=0}^{p-1}{\rm Card}\Big(E_{k-2}^{\frac{i}{p^{k}}+\frac{j}{p^{k-1}}}\Big)={\rm Card}\Big(E_{k-1}^{\frac{i}{p^k}}\Big),$$ by (\ref{ud1}), we get
$$
{\rm Card}\Big(E_{k-2}^{\frac{i}{p^{k}}+\frac{j}{p^{k-1}}}\Big)
={\rm Card}\Big(E_{k-2}^{\frac{l}{p^{k}}+\frac{m}{p^{k-1}}}\Big), \quad   \forall\  0\le i, j, l,m \le p-1.
$$

We continue these arguments    for all $\eta= p^{k-1}, \cdots, p^{n_f}$. 
By induction, we have  
\begin{align}\label{equalnumber}
{\rm Card}(E_{n_f}^{\xi_1})
={\rm Card}(E_{n_f}^{\xi_2}), \quad   \forall \  \xi_1,  \xi_2 \in \mathbb{L}_{n_f} \cap B(0,p^k).
\end{align}
Since $|\xi|_p=p^{k}$, there exists $\xi^{\prime}\in \mathbb{L}_{n_f} \cap B(0,p^k)$, such that  
$E_{n_f}^{\xi}=E_{n_f}^{\xi^{\prime}}$. Thus  by (\ref{equalnumber}),  
$${\rm Card}(E_{n_f}^{\xi})
={\rm Card}(E_{n_f}^{\xi^{\prime}})={\rm Card}(E_{n_f}).$$ 

The formula (\ref{numberE}) follows immediately because
each ball of radius  $p^n$ with $ n\ge n_f$ is a disjoint union of  $p^{n-n_f}$ balls of radius $p^{n_f}$ so that
$$
   {\rm Card} (E_n) = p^{n-n_f} {\rm Card} (E_{n_f}).
$$  
\end{proof}

\subsection {Equality $D(E)=1/ \int_{\Q_p} f d\m$}
\begin{prop}\label{fisrtprop}
The  density $D(E)$ of $E$ satisfies
 $$D(E)=\frac{1}{\int_{\Q_p}f(x)dx}.
 $$
\end{prop}

\begin{proof}
By the integrability of $f$, the quantity
$$
\epsilon_n := \int_{\Q_p\setminus B(0, p^n)} f(x) dx
$$
tends to zero  as  the integer $n\to \infty$.
	Integrating the equality (\ref{EQ}) over the ball $B(0, p^n)$, we have
\begin{eqnarray}\label{ee}
  \m (B(0,p^n))
		=\sum_{\lambda \in E} \int_{B(0,p^n)} f(x-\lambda)dx.  
\end{eqnarray}
Now we split the sum in (\ref{ee}) into two parts, according to $\lambda \in E\cap B(0, p^n)$ or $\lambda \in E\setminus B(0, p^n)$.  Denote $I:=\int_{\Q_p} f d\m$.
For $\lambda \in E\cap B(0, p^n)$,  we have
\begin{eqnarray*}
\int_{B(0,p^n)} f(x-\lambda)dx =	\int_{B(0,p^n)} f(x)dx
	= I - \epsilon_n.
\end{eqnarray*}
It follows that 
\begin{equation}\label{ee1}
\sum_{\lambda \in E\cap B(0, p^n)}	\int_{B(0,p^n)} f(x-\lambda)dx
=   \mbox {\rm Card} (E_n ) \cdot (I-\epsilon_n).
\end{equation}

Notice that 
\begin{equation}\label{ee2}
    \int_{B(0,p^n)} f(x-\lambda)dx
    = \int_{B(-\lambda,p^n)} f(x)dx.
\end{equation}
For $\lambda \in E\setminus B(0, p^n)$, the ball  $B(-\lambda,p^n)$ is contained in $\Q_p\setminus B(0, p^n)$.
We partition the uniformly discrete set $E\setminus B(0, p^n)$ into $P_j$'s such that each $P_j$ is contained in a ball
of radius $p^n$ in $\Q_p\setminus B(0, p^n)$. Thus the integrals in (\ref{ee2}) for the $\lambda$'s in the same $P_j$ are equal.
 Let $\lambda_j$ be a representative of $P_j$. Then we have
\begin{eqnarray}
\sum_{\lambda \in E\setminus B(0,p^n)} \int_{B(0,p^n)} f(x-\lambda)dx
 =  \sum_j \mbox{\rm Card} (P_j) \cdot \int_{B(-\lambda_j,p^n)} f(x)dx. \nonumber
 \end{eqnarray}
 However, by (\ref{numberE}) in Proposition \ref{Structure},  $\mbox{\rm Card} (P_j) = D(E) \m(B(0, p^n))$  if $n \geq n_f$. Thus, for each  integer $n\geq n_f$,
 \begin{align}
 \sum_{e\in E\setminus B(0,p^n)} \int_{B(0,p^n)} f(x-\lambda)dx
       & =  D(E)\m(B(0, p^n)) \sum_j \int_{B(-\lambda_j,p^n)} f(x)dx  \nonumber \\
       & =  D(E) \m(B(0, p^n)) \int_{\Q_p \setminus B(0, p^n)} f(x) dx  \nonumber  \\
       &=   D(E) \m(B(0, p^n)) \epsilon_n. \label{ee3}
\end{align}

Thus from (\ref{ee}), (\ref{ee1}) and (\ref{ee3}),  we finally get
	$$
	\left|  \m(B(0, p^n)) - \mbox{\rm Card}(E\cap B(0, p^n))\cdot I \right|
      \le 2 D(E) \m(B(0, p^n)) \epsilon_n.
	$$
We conclude by dividing $\m(B(0, p^n))$ and then letting $n\to \infty$.
\end{proof}

\setcounter{equation}{0}
\section{Tiles are spectral sets}
The key for the proof of the following theorem
is the fact that the equation $\mu_E *f =1$ imposes that the support of the Fourier transform $\widehat{f}$ is compact (Proposition \ref{CompactSupport}). We also use the fact that  the Fourier transform of an integrable function with compact support is
uniformly locally constant (Proposition \ref{2case}). The Lebesgue density theorem (Proposition \ref{Lebesgue}) is used too. 
We recall that a point $x$ in a Borel set $\Omega$ satisfying the equality in Proposition \ref{Lebesgue} is called a {\em density  point} of $\Omega$.


\begin{thm}\label{propmain} Without  distinguishing sets which are equal modulo a set of zero Haar measure, 
	we have the following assertions :
	\begin{itemize}
		\item[(1)] If $\Omega$ is a spectral set in $\mathbb{Q}_p$, then it is bounded.
		\item[(2)] If $\Omega$ is a tile in $\mathbb{Q}_p$, then it is compact and open.
		\item[(3)] Tiles in $\Qp$ are spectral sets. 
	\end{itemize}
\end{thm}
\begin{proof}
Since we do not distinguish sets which are equal modulo a set of zero Haar measure,   we can assume that all the points in $\Omega$ are density points of $\Omega$, according to the Lebesgue density theorem (Proposition \ref{Lebesgue}).

(1) Assume that $(\Omega, \Lambda)$ is a spectral pair in $\Q_p$, which is equivalent to
 $$
 \forall  \xi\in \widehat{\Qp}, \quad 
 \mu_\Lambda*|\widehat{1_\Omega}|^2 (\xi) =\m(\Omega)^2.$$
 By Proposition \ref{CompactSupport},  $\widehat{|\widehat{1_\Omega}|^2}$ has compact support.  Observe that  
 \begin{equation}\label{FFO}
\widehat{|\widehat{1_\Omega}|^2}(\xi)=1_\Omega * 1_{-\Omega} (\xi)= \int_{\mathbb{Q}_p} 1_\Omega(x)1_\Omega(x -\xi)dx=\m(\Omega\cap (\Omega+\xi)).
 \end{equation}
We claim
$$
\operatorname{supp}(\widehat{|\widehat{1_\Omega}|^2})=\overline{\Omega-\Omega} .
$$
Since the inclusion $\operatorname{supp}(\widehat{|\widehat{1_\Omega}|^2})\subset\overline{\Omega-\Omega}$ is obvious, we need only to show  
\begin{equation}
\label{O-O}
\overline{\Omega-\Omega} \subset \operatorname{supp}(\widehat{|\widehat{1_\Omega}|^2}).
\end{equation}
In fact, let $\xi\in \Omega-\Omega$.  Write  $\xi=z_1-z_2$ with  $z_1,z_2\in \Omega$. By (\ref{FFO}), we have
\begin{align*}
	\widehat{|\widehat{1_\Omega}|^2}(\xi)
	=\m(\Omega\cap(\Omega+z_1-z_2)) =\m((\Omega-z_1)\cap(\Omega-z_2)).
\end{align*}
Since $z_1,z_2$ are density  points of $\Omega$,  $0$ is a  density point of both $\Omega-z_1$ and $\Omega-z_2$.  This fact  implies 
$$\m((\Omega-z_1)\cap(\Omega-z_2)>0.$$
Thus we have proved $\Omega-\Omega\subset \operatorname{supp}(\widehat{|\widehat{1_\Omega}|^2})$.
Then (\ref{O-O}) follows. 
Since $\operatorname{supp}(\widehat{|\widehat{1_\Omega}|^2})$ is compact,  the set $\Omega$ is bounded.

(2) The main argument is the same as in (1). Assume that $(\Omega,T)$ is a tiling pair  in $\Qp$, which means 
 $$\mu_{T}*{1}_\Omega(x) =1, \quad \m\text {-a.e. } x\in\Qp.$$ By Proposition \ref{CompactSupport}, $\widehat{1_\Omega}$ has compact support. Then by  the first  assertion of Proposition  \ref{2case},  $1_\Omega$ is  almost uniformly locally constant, i.e. $\Omega$ is, up to a zero measure set, a union of balls with the same radius. Since $\Omega$ is of finite measure,  the number of these balls is finite.  So,  $\Omega$ is  almost compact open.

(3) It is an immediate  consequence  of   (2) and  Theorem \ref{compactopen}. 
\end{proof}

The above proof of the fact ``tiles are spectral sets" is partially based on Theorem \ref{compactopen} and partially on ``tiles are compact and open" (Theorem \ref{propmain} (2)).  
The proof of ``spectral sets are tiles" will be not based on 
 Theorem \ref{compactopen}. That means, we are not going to show
 that a spectral set is a compact open set up to a set of zero measure. But the boundedness of a spectral set (Theorem \ref{propmain} (1)) will be used.

\setcounter{equation}{0}
\section{Spectral sets are tiles}
In this section, we prove that a spectral set  in $\Qp$ is  a tile. 
Assume that $(\Omega,\Lambda)$ is a spectral pair.
The proof will be based on our knowledge on the set 
$\mathcal{Z}_{\widehat{\mu}_\Lambda}$. As we will see in the proof,
	a tiling complement can be easily constructed.


\begin{thm}\label{SpectralTile}
If $\Omega$ is a spectral set in $\mathbb{Q}_p$, then it is a tile.
\end{thm}

\begin{proof}
	Suppose that $(\Omega, \Lambda)$ is a spectral pair. That means (see (\ref{SpectralSet})) $$\sum_{\lambda\in\Lambda}|\widehat{1_\Omega}|^2(x-\lambda)=\m(\Omega)^2, \quad   \m\hbox {-a.e.~} x\in \Qp.$$
	In other words,  $\mu_\Lambda * f =1$ where 
	$$f=|\widehat{1_\Omega}|^2/\m(\Omega)^2.$$
		Let $n_f$ be the integer defined by (\ref{nf}). Recall that
		$B(0, p^{-n_f})$
		is the biggest ball  centered at $0$ over which $\widehat{f}$ does not vanish. We first remark that 
		\begin{equation}\label{densitySpectral}
		\forall n\ge n_f, \quad 	{\rm Card} (\Lambda_{n})= p^n \m(\Omega)
		\end{equation}
		where
		$$\Lambda_{n}:=\Lambda\cap B(0,p^n).
		$$
In fact, 	Plancherel Theorem implies $\int_{\Q_p} f d\m =1/\m(\Omega)$. 
	So, by  Proposition \ref{fisrtprop},   the set $\Lambda$ is  of bounded density with density
	$$
	D(\Lambda)=\frac{1}{ \int_{\Qp} fd\m}=\m(\Omega).
	$$	
Thus (\ref{densitySpectral}) follows from  (\ref{numberE}) in Proposition \ref{Structure}.  
	

By Theorem \ref{propmain}, the spectral set 	
 $\Omega$ is  bounded, up to a Haar null measure set. 
Without loss of generality,  we assume that $\Omega\subset \Z_p$.


Recall that $\mathcal{Z}_{\widehat{\mu_\Lambda}}$ has the following properties: every sphere $S(0,p^{-n})$ either is contained in 
$\mathcal{Z}_{\widehat{\mu_\Lambda}}$ or does not intersect 
$\mathcal{Z}_{\widehat{\mu_\Lambda}}$ (Lemma \ref{zeroofE}); 
	all spheres $S(0,p^{-n})$ with  $n\geq n_f$
	are contained in $\mathcal{Z}_{\widehat{\mu_\Lambda}}$
	(see (\ref{3.1})). 
	Let
	\begin{eqnarray*}	
\mathbb{I}:&= &\left\{  0\leq n<n_f:  S(0,p^{-n})\subset\mathcal{Z}_{\widehat{\mu_\Lambda}} \right\}, \\ 
	\mathbb{J}:&=&\{ 0\leq n < n_f: S(0,p^{-n})\cap \mathcal{Z}_{\widehat{\mu_\Lambda}}=\emptyset\} .
		\end{eqnarray*}
	For $\mathcal{N}_{\widehat{f}}$, we claim 
	\begin{equation}\label{StoT1}
	\bigcup_{j\in \mathbb{J}} S(0,p^{-j})\subset \mathcal{N}_{\widehat{f}}.
	\end{equation}
	In fact, otherwise  for some $\xi\in S(0,p^{-j})$ with $j\in \mathbb{J}$ we have $\widehat{f}(\xi)\ne 0$. By (\ref{mainprop}) and Proposition \ref{zeroofE} (1),
		 $S(0,p^{-j}) \subset \mathcal{Z}_{\widehat{\mu_\Lambda}}$. 
		 This  contradicts the fact  $j\in \mathbb{J}$.

Observe that $\widehat{f}(\xi) = \m(\Omega \cap (\Omega+\xi))$ (see (\ref{FFO})) and that 
$$
(U-U)\setminus\{0\}\subset \bigcup_{j\in \mathbb{J}} S(0,p^{-j})$$
where $$
U:=\left\{\sum_{j\in \mathbb{J}} \alpha_j p^{j}, \alpha_j\in \{ 0, 1, \dots, p-1 \}  \right\}.
$$
Then from (\ref{StoT1}), we get
 \begin{equation}\label{StoT2}
 (U-U)\setminus\{0\}
     \subset \{\xi\in \Q_p: \m(\Omega \cap (\Omega+ \xi)=0\}.
 \end{equation}

We claim that  
	$\Omega$ is a tile of  $\Z_p$ with tiling complement $U$, so that $\Omega$ is a tile of  $\Q_p$ with tiling complement $U + \mathbb{L}$.
	
In fact, the disjointness (up to a set of zero measure) of $\Omega + \xi_1$ and $\Omega + \xi_2$ for any two distinct $\xi_1, \xi_2 \in U$ follows directly from (\ref{StoT2})
and the invariance of the Haar measure.

So to show that $\Omega$ is a tile of  $\Z_p$, it suffices to prove that the total measure of all $\Omega + \xi$ with
$\xi \in U$ is equal to 
\begin{align}\label{measure1}
 \m(\Omega+U)={\rm Card}(U)\cdot \m(\Omega)  = 1.
\end{align}
  

For $i \in \mathbb{I}$, we have $B(p^{-i}, p^{-n_f})\subset S(0, p^{-i})$ so  that 
 $B(p^{-i},p^{-n_f}) \subset \Zm$, which  implies 
\begin{align*}
		0=\langle \widehat{\mu_\Lambda}, 1_{B(p^{-i},p^{-N})} \rangle
		=\langle \mu_\Lambda, \widehat{1_{B(p^{-i},p^{-n_f})}} \rangle
		=p^{-n_f}\sum_{\lambda\in \Lambda_{n_f}} \overline{\chi(p^{-i} \lambda)}.
	\end{align*}
By Lemma \ref{zmod} (3),  $$p^{{\rm Card}(\mathbb{I})}\leq  {\rm Card} (\Lambda_{n_f}).$$
Then by the fact ${\rm Card}(U)=p^{{\rm Card}(\mathbb{J})}$ and (\ref{densitySpectral}),  we have
	$$
	{\rm Card}(U) \cdot \m(\Omega)
	= p^{{\rm Card}(\mathbb{J})} \cdot 
	\frac{{\rm Card} (\Lambda_{n_f})}{p^{n_f}} 
	\ge
	p^{{\rm Card}(\mathbb{J})} \cdot 
	\frac{p^{{\rm Card}(\mathbb{I})}}{p^{n_f}} 
=1.
	$$
We can now conclude (\ref{measure1}) because 
$\Omega +U$ is contained in $\mathbb{Z}_p$ so that
$\m(\Omega +U)\le 1$.

\end{proof}


\setcounter{equation}{0}
\section{Some remarks }\label{remarks}
\subsection{Dual Fuglede's conjecture}
Theorem \ref{mainthm} asserts  that all spectral sets (equivalently tiles) in $\Q_p$ are almost compact open. What is the topological structure of the corresponding spectra and tiling complements? This question has been answered by Theorem 1.2 of \cite{FFS} for the compact open spectral sets in $\Q_p$. 
Recall that  
the spectra and tiling complements of compact open spectral sets are infinite 
$p$-homogeneous  discrete sets $\Lambda$ (see \cite[Section 2.8]{FFS}) such  that    there exists an integer  $N$ such that 
\begin{align}\label{homogeneous}
  (\infty, N]\cap \Z  \subset  I_{\Lambda},  
\end{align}
where
$I_{\Lambda}:=\{v_p(x-y): x,y \in \Lambda  \text{ distinct}\}$,  called the set of admissible  $p$-orders of $\Omega$. 
  On the other hand, for any given infinite 
$p$-homogeneous  discrete sets $\Lambda$ satisfying (\ref{homogeneous}), we can construct two compact open sets $\Omega_1$ and $\Omega_2$ such that $(\Omega_1,\Lambda)$ is a spectral pair and $(\Omega_2,\Lambda)$ is a tiling pair. For details, we refer the readers to \cite[Theorems 1.2  and 5.1]{FFS}.
 So  the dual Fuglede's conjecture  holds in $\Q_p$ (see \cite{jp} for the dual Fuglede's conjecture in $\mathbb{R}^d$):
   {\em a subset of $\Qp$ is a spectrum if and only if it is a tiling complement.} 
\subsection{The density of $E$ satisfying  (\ref{F-equation}) in  $\Q_p^d$ }
Assume that   $E$ is a discrete subset in $\Q_p^{d}$ such that ${\rm Card}(E\cap K)<\infty$  for any compact subset $K$ of $\Q_p^d$.
Then  $\mu_E = \sum_{\lambda\in E} \delta_\lambda $
 defines a   discrete
 Radon measure in $\Q_p^d$. Remark that  $E$ is not necessarily assumed uniformly discrete. 
  Suppose that  $E$   satisfies   the  equation (\ref{F-equation}) with $f \in L^1(\Q_p^d)$ which is a non-negative integrable function such that $\int_{\Q_p^d}f d\m>0$.
  Then for each $x_0\in \Q_p^d$, the limit  
  $$D(E)=\lim\limits_{k\to \infty}\frac{{\rm Card}(B(x_0,p^k)\cap E)}{\m (B(x_0,p^k))}$$
  exists and 
 $$D(E)=\frac{1}{\int_{\Q_p^d}f d\m}.$$
This  is similar to the real case, see  \cite[Lemma 2]{Kol}.
 
\subsection{Tiles and spectral sets in $\Q_p^2$  which are not compact open.}
For a uniformly discrete set $E$ in the higher dimensional space $\Q_p^d$ with $d\geq 2$, the zero set of the Fourier transform of the measure $\mu_E$ is not necessarily  bounded.  In other words,
Proposition \ref{zeroofE} does not hold in $\Q_p^d$ with $d\geq 2$.   For example,  let $E=\{(0,0),(0,1),\cdots,(0,p-1) \}$ which is a finite subset of $\Q_p^2$. One can check that 
$$\ZE= \Q_p \times p^{-1} \Z_p^{\times}.
$$

Bounded tiles of $\Q_p^2$ are not necessarily almost compact open. 
Let us construct such a bounded tile of $\Q_p^2$. 
We partition $\Z_p$ into $p$ Borel sets of same Haar measure,
denoted $A_i$ $(i = 0, 1, ..., p-1)$. We assume that one of $A_i$ is not almost compact open.   For example, set $S= \bigcup_{n=1}^{\infty} B(p^n,p^{-n-1})$, which is a  union of countable disjoint balls.
Let 
\begin{eqnarray*}
A_0 &= & S \cup (B(1,p^{-1}) \setminus  (1+S)),  \\ 
A_1 &= &(B(0,p^{-1})\cup B(1,p^{-1}) )\setminus A_0,\\
A_i &= & B(i,p^{-1})\ \ \mbox{\rm for} \ \ 2\leq i\leq p-1. 
\end{eqnarray*}
 Then define $$
     \Omega:= \bigcup_{i= 0}^{p-1}  A_i \times B(i, p^{-1})\subset \Z_p\times \Z_p.
     $$
      The set $\Omega$  is not almost compact open, because any small ball centered at $(0,0)$ meets both  $\Omega$ and $(\Z_p\times \Zp\setminus \Omega)$ with positive measure. 
      But it is a tile of $\Q_p^{2}$
      with tiling complement
$$T=\mathbb{L}\times \mathbb{L}_{-1}.$$
Just like on $\mathbb{R}^d$(see \cite{F}), it can be  
         proved that $(\Omega, \mathbb{L}\times \mathbb{L}_{-1})$ is a tiling pair if and only if $(\Omega, \mathbb{L}\times \mathbb{L}_{1})$ is a spectral pair (see \cite{Kad}).

\subsection{Higher dimensional cases}
However, the situation changes in higher dimension. In fact,  
 it is  proved in \cite{Iso} that  there  exist  spectral sets which are not  tiles in $\mathbb{F}_p^5$ for all odd primes $p$ and in $\mathbb{F}_p^4$ for all odd primes $p$ such that $p\equiv 3 \mod 4$,
where $\mathbb{F}_p$ is the prime field with $p$ elements.   This implies that there  exist  compact open spectral sets which are not  tiles in $\mathbb{Q}_p^5$ for all odd primes $p$ and in $\mathbb{Q}_p^4$ for all odd primes $p$ such that $p\equiv 3 \mod 4$.   

For $d=3$, there  exists a spectral set in  the finite group $(\Z/8\Z)^3$   which is not  a tile \cite{KM2}. This implies that  there  exists   a compact open spectral set which is not  a tile in $\mathbb{Q}_2^3$.

 Fuglede's conjecture in  $\Q_p^2$ is open, even under the compact open assumption.

\end{document}